\documentclass[12pt]{article}

\usepackage{ufrr-template}

\usepackage{graphicx,url}

\usepackage[latin1]{inputenc}  
\usepackage{eso-pic}
\usepackage{color}
\usepackage{fancybox}

\usepackage{amsmath,amsfonts}
\usepackage{amsthm}
\usepackage{amssymb}
\usepackage{amscd,bezier}
\usepackage{hyperref}
\usepackage[latin1]{inputenc}
\usepackage[T1]{fontenc}
\usepackage{graphicx}
\usepackage{placeins}
\usepackage[active]{srcltx}
\usepackage{breqn}
\definecolor{myblue}{RGB}{41,87,163}
\definecolor{myblue1}{RGB}{50,30,200}

\newtheorem{defi}{Definition}
\newtheorem{lema}{Lemma}

\newtheorem{proposicao}{Proposition}
\newtheorem{theorem}{Theorem}

\makeatletter
\def\moverlay{\mathpalette\mov@rlay}
\def\mov@rlay#1#2{\leavevmode\vtop{%
		\baselineskip\z@skip \lineskiplimit-\maxdimen
		\ialign{\hfil$\m@th#1##$\hfil\cr#2\crcr}}}
\newcommand{\charfusion}[3][\mathord]{
	#1{\ifx#1\mathop\vphantom{#2}\fi
		\mathpalette\mov@rlay{#2\cr#3}
	}
	\ifx#1\mathop\expandafter\displaylimits\fi}
\makeatother

\newcommand{\cupdot}{\charfusion[\mathbin]{\cup}{\cdot}}

\makeatletter
\AddToShipoutPicture{%
\setlength{\@tempdimb}{10.5cm}%
\setlength{\@tempdimc}{27cm}%
\setlength{\unitlength}{1pt}%
\put(\strip@pt\@tempdimb,\strip@pt\@tempdimc){%
\makebox(0,0){\rotatebox{0}{\includegraphics[keepaspectratio=false,scale=0.92]{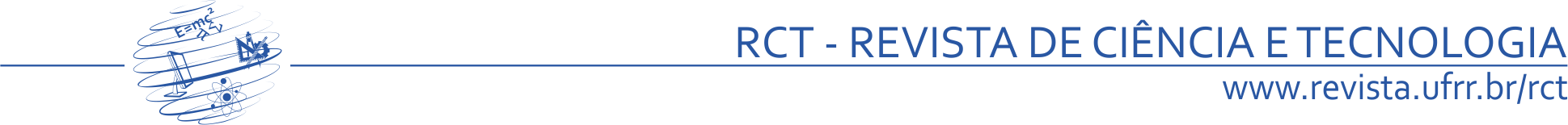}%
}}%
}%
\setlength{\@tempdimb}{4.75cm}%
\setlength{\@tempdimc}{1.5cm}%
\setlength{\unitlength}{1pt}%
\put(\strip@pt\@tempdimb,\strip@pt\@tempdimc){%
\makebox(0,0){\rotatebox{0}%
{\textcolor{myblue}{RCT} \textcolor{myblue}{v.1.n.1 (2016)}}
}%
}%
\setlength{\@tempdimb}{10.55cm}%
\setlength{\@tempdimc}{1.5cm}%
\setlength{\unitlength}{1pt}%
\put(\strip@pt\@tempdimb,\strip@pt\@tempdimc){%
\makebox(100,0){\rotatebox{0}%
{\textcolor{myblue}{\rule{8.2cm}{0.5pt}}{\resizebox{3.3cm}{0.31cm}{\textcolor{myblue}{\ ISSN 2447-7028}}}}
}%

}

}
\makeatother

\title{The Fubini Theorem for Normal Lie Subgroups of Index $2n$}

\author{Leandro Nery de Oliveira\inst{1}, Marcos Aurélio de Alcântara\inst{2}}

\address{Departamento de Matemática -- Universidade Federal de São Carlos
	(UFSCar)\\
	São Carlos -- SP -- Brazil
	\nextinstitute
	Centro de Ciências Exatas e Tecnológicas -- Universidade Federal do Acre (UFAC) \\
	Rio Branco -- AC -- Brazil
	\email{leandro.oliveira@ufscar.br, 	marcos.alcantara@ufac.br}
}

\begin{document} 

\maketitle

\begin{abstract}
  Let $\Gamma_+$ be a normal subgroup of index $2n$ of a group $\Gamma$ and $\gamma_i \in \Gamma \setminus \Gamma_+$ be involutions. We first prove that if $\Gamma = \Gamma_+ \rtimes (\mathbb{Z}_2(\gamma_1) \times \cdots \times \mathbb{Z}_2(\gamma_n))$ then $\Gamma = (\Gamma_+ \rtimes \mathbb{Z}_2(\gamma_1) \rtimes \cdots \rtimes \mathbb{Z}_2(\gamma_{i-1})) \rtimes (\mathbb{Z}_2(\gamma_{i}) \times \cdots \times \mathbb{Z}_2(\gamma_n))$, where $i=2,\cdots,n$. Second, we use this result to prove the well-known Fubini theorem for a subgroup of index $2n$ of a compact Lie group. Finally, we present an application to invariant theory.
\end{abstract}


\section{Introduction}


The Fubini theorem was introduced by Guido Fubini in 1907 (\cite{fubini1907}). This theorem presents many variants and reduces integration in multiple variables to more simple iterated integrals. As a consequence, the integration order can be reversed in iterated integrals. 

In this paper, we work with definite integrals over a compact Lie group $\Gamma$. These integrals are called Haar integrals and have the property of being invariant under translation by elements of the group. In this case, the Fubini theorem presents a version that is applied to subgroups of Lie group that have index $2$, under the conditions that $\Gamma$ is a compact Lie group, $\Gamma_+ \subset \Gamma$ a subgroup of index $2$ of $\Gamma$ and $\Gamma_-=\Gamma \setminus \Gamma_+$ the complement set of $\Gamma_+$ with respect to $\Gamma$. Then for any $f: \Gamma \to \mathbb{R}$ continuous we have
$$\int_{\Gamma}f(\gamma)=\frac{1}{2} \left(\int_{\Gamma_+}f(\gamma)+\int_{\Gamma_+}f(\lambda \gamma) \right),$$
for fixed $\lambda \in \Gamma_-$. Of course, if $\lambda \in \Gamma_+$ the equality is not valid, since
$$\int_{\Gamma}f(\gamma) \neq \int_{\Gamma_+}f(\gamma),$$ 
because $\Gamma_+ \subsetneq \Gamma$. Also, note that $\Gamma_-$ is a homogeneous space.
This version of the theorem has applications in recent work, one of which we can mention \cite{zeli2013}.

Our motivation to study the Fubini theorem came from the Lorentz orthogonal group. This group has four connected components and presents important properties in invariant theory (\cite{oliveiraaspectos} and \cite{manoel2019equivariant}) and in Einstein's relativity physics (\cite{einstein1905} and \cite{minkowski1909}). The paper of \cite{carmeli1986rotational} illustrates the application of Lorentz orthogonal group in relativity physics. Since the Lorentz group has four connected components, it can be written as the semi-direct product normal subgroup of index $4$ with the cartesian product of two involution subgroups. Therefore, it is often necessary to apply Fubini theorem to the integral calculation in normal subgroups of index larger than $2$.

This paper is organized as follows: In Section 2, we present the basic concepts in which we can approach this theorem as the representation theory of Lie group and Haar integral. We begin Section 3 by presenting an important algebraic result, namely Proposition \ref{corsemidireto}:

{ \it
Let $\Gamma_+$ be a normal subgroup of index $2n$ of $\Gamma$ and $\gamma_i \in \Gamma \setminus \Gamma_+$ involutions, $i \in \{1,\cdots,n\}$. If $\Gamma$ can be decomposed as $$\Gamma = \Gamma_+ \rtimes (\mathbb{Z}_2(\gamma_1) \times \cdots \times \mathbb{Z}_2(\gamma_n)),$$ then $$\Gamma = (\Gamma_+ \rtimes \mathbb{Z}_2(\gamma_1) \rtimes \cdots \rtimes \mathbb{Z}_2(\gamma_{i})) \rtimes (\mathbb{Z}_2(\gamma_{i+1}) \times \cdots \times \mathbb{Z}_2(\gamma_n)),$$ 
where $i \in \{1,\cdots,n\}$.}

This result allows us to make a new proof of the Fubini Theorem for normal Lie subgroups of index $2n$, Theorem \ref{genfubini}:

{ \it
	Let $\Gamma$ be a compact Lie group and $\Gamma_+ \subset \Gamma$ a normal subgroup of index $2n$ of $\Gamma$. For any continuous $f: \Gamma \to \mathbb{R}$ we have
	$$\int_{\Gamma}f(\gamma)=\frac{1}{2^n} \left(\sum_{i \in N^n_2} \int_{\Gamma_+}f(\lambda^i\gamma) \right),$$
	for $\lambda_i \in \Gamma_-$ fixed.
}

We conclude this paper by showing an application of this theorem in invariant theory, specifically for the calculation of the Molien series.

\section{Preliminaries}

In this section we show some basic concepts for understanding the Fubini theorem. We begin with a definition of representation and group actions and define the Haar integral. We conclude this section by presenting a classic version of the Fubini theorem for group actions.

\subsection{Representations and group actions}

Here we see a notion about Lie groups and representation theory. Let start with the Lie group definition. More details on this topic can be found at \cite{hall2015lie}. 

A \emph{Lie group} $\Gamma$ is a differentiable manifold with a group structure in which the multiplication and inversion maps 
\begin{center}
\begin{tabular}{ccc}
	\begin{tabular}{ccc}
		$\Gamma \times \Gamma$ & $\to$ & $\Gamma$\\
		$(\gamma,\delta)$&$\mapsto$& $\gamma\delta$
	\end{tabular}
	& $\quad$ and $\quad$ & 
	\begin{tabular}{ccc}
		$\Gamma$ & $\to$ & $\Gamma$ \\
	 $\gamma$&$\mapsto$& $\gamma^{-1}$
	\end{tabular}
\end{tabular}
\end{center}
are smooth.

Cartan's closed subgroup theorem asserts that any closed subgroup $\Gamma_+$ of a Lie group $\Gamma$ is a Lie subgroup. This theorem admits the following converse: If $\Gamma_+ \subseteq \Gamma$ is an embedded Lie subgroup, then $\Gamma_+$ is closed. For more details, proofs and examples, see \cite{san2021lie}, page 129. A Lie group is  compact if it is compact as a topological space.

	
 
 
 Throughout this work, the group $\Gamma$ represents a compact Lie group and $V$ a vector space, unless otherwise noted.

We define the\emph{ left action of a Lie group $\Gamma$ on a vector space} $V$, or simply an action of $\Gamma$ on $V$, as a function $\varphi : \Gamma \times V \to V$ such that for all $\gamma, \delta \in \Gamma$ and $v \in V$ we have $$ \varphi (\gamma,\varphi(\delta,v))=\varphi(\gamma\delta,v) \quad \text{and} \quad \varphi(1,v)=v,$$
where $1$ is the identity element of $\Gamma$. To abbreviate the notation, we use $\gamma v$ instead of $\varphi(\gamma,v)$ to indicate the action of $\gamma$ on $v$. 

An action of $\Gamma$ on $V$ corresponds to a group homomorphism 
\begin{eqnarray*}
\rho &:& \Gamma \longrightarrow GL(n) \\
& &\gamma \longmapsto \rho_{\gamma}
\end{eqnarray*}
called the representation of $\Gamma$ on $V$.


\subsection{The Haar integral and the Fubini theorem}

In this section we present the Haar integral and discuss the Fubini theorem for Lie groups.

We can identify any compact Lie group contained in $GL(n)$ as a subgroup of the orthogonal group $O(n)$. This identification is made using a form of integration that is invariant by translating the elements of $\Gamma$, called the \textit{Haar integral}. 

A proper definition of the Haar integral must be based on the Haar measure. However, for our purposes, it is sufficient to consider the Haar integral as an operation that satisfies the properties given in Definition \ref{haar}. For more details on the Haar integral, see \cite{nachbin1976haar}.

\bigskip

\begin{defi}[\cite{golubitsky2012singularities}] \label{haar}
	 Let $f:\Gamma \to \mathbb R$ be a continuous real function. The operation $\int_{\gamma \in \Gamma} f(\gamma)$ or $\int_{\Gamma} f$ it is a Haar integral in $\Gamma$ if it satisfied the following properties:
	 \begin{itemize}
	 	\item[(i)]	Linearity: $$\int_{\Gamma} (af+bg)=a \int_{\Gamma}f+b\int_{\Gamma}g,$$
	 	where $f,g:\Gamma \to \mathbb R$ are continuous functions and $a,b \in \mathbb R$.
	 	\item[(ii)] Positivity: If $f(\gamma) \geq 0$, for all $\gamma \in \Gamma$ then $$\int_{\Gamma} f \geq 0.$$
	 	\item[(iii)] Translational invariance on the left:
	 	$$\int_{\gamma \in \Gamma} f(\delta \gamma)=\int_{\gamma \in \Gamma} f(\gamma),$$
	 	for any fixed $\delta \in \Gamma$.
	\end{itemize}
\end{defi}

\bigskip

It is known that the Haar integral is unique \cite{hochschild1965structure}. If the Lie group is compact then the Haar integral is also invariant by right-hand translation and can be normalized, that is, $\int_{\Gamma} 1 = 1$.

If $\Gamma$ is a finite Lie group then the Haar integral on $\Gamma$ is 
$$\int_{\Gamma}f(\gamma) \equiv \frac{1}{|\Gamma|} \sum_{\gamma \in \Gamma} f(\gamma).$$
Let $\Gamma$ be a compact Lie group isomorphic to the special orthogonal group $SO(2)$, for every continuous function $f:SO(2) \to \mathbb{R}$ the Haar integral on $\Gamma$ is 
\begin{eqnarray} \label{haarcompacto}
	\int_{\Gamma}f(\gamma) \equiv \frac{1}{2 \pi}\int_{0}^{2 \pi} f(R_{\theta})d\theta,
\end{eqnarray}
where $R_{\theta}$ is a rotation of $SO(2)$, with $0 \leq \theta \leq 2\pi$.

In general, it is not easy to calculate Haar integral, so the Fubini theorem enunciated below, reduces the calculation of the Haar integral on $\Gamma$ for integrals on $\Gamma_+$, a subgroup of index $2$ of $\Gamma$.

\bigskip

\begin{theorem}[The Fubini theorem] \label{fubini}
	Let $\Gamma$ be a compact Lie group and $\Gamma_+ \subset \Gamma$ a subgroup of index $2$ of $\Gamma$. For any continuous function $f: \Gamma \to \mathbb{R}$, we have
	$$\int_{\Gamma}f(\gamma)=\frac{1}{2} \left(\int_{\Gamma_+}f(\gamma)+\int_{\Gamma_+}f(\lambda \gamma) \right),$$
	for fixed $\lambda \in \Gamma_-$.
\end{theorem}

\bigskip

Note that since $\Gamma_+$ is a subgroup of index $2$ of $\Gamma$, then a $\Gamma$ Lie group can be written as a semi-direct product of $\Gamma_+$ and $\mathbb{Z}_2(\lambda)$, i.e.,
$$\Gamma=\Gamma_+ \rtimes \mathbb{Z}_2(\lambda),$$
where $\mathbb{Z}_2(\lambda)$ is a subgroup of $\Gamma$ generated by an involution $\lambda \in \Gamma_-$. Recall that the semi-direct product of groups is the product of subgroups, where $\Gamma_+$ is a normal subgroup of $\Gamma$ and the subgroups $\Gamma_+$ and $\mathbb{Z}_2(\lambda)$ have trivial intersection (\cite{robinson2012course}).

\section{The Fubini theorem for Lie normal subgroups of index $2n$}
In this section we present the Fubini theorem for the case where the group $\Gamma_+$ is a normal subgroup of index $2n$ of $\Gamma$. We also present an application to invariant theory. For this we need the following lemmas.  

\bigskip

\begin{lema}[\cite{oliveiraaspectos}]\label{semidireto}
	Let $\Gamma$ be a Lie group and $\Gamma_+$ be a normal subgroup of $\Gamma$. Suppose that $$\Gamma=\Gamma_+ \rtimes (\mathbb{Z}_2(\gamma_1) \times \mathbb{Z}_2(\gamma_2)),$$ where $\gamma_1, \gamma_2 \in \Gamma \setminus \Gamma_+$ are distinct involutions. Then $$\Gamma=(\Gamma_+ \rtimes \mathbb{Z}_2(\gamma_1)) \rtimes \mathbb{Z}_2(\gamma_2).$$
\end{lema}

\bigskip

\begin{lema} \label{z2abeliano}
	The group $\Gamma=\Gamma_1 \times \Gamma_2 \times \cdots \Gamma_n$ is abelian if and only if $\Gamma_i$ is abelian, with $i=1,2,\cdots,n$.
\end{lema}

In the following proposition, we generalize the Lemma \ref{semidireto} for the case where $ \Gamma_+ $ is a normal subgroup of index $2n$ of $ \Gamma$. The proof is similar to the proof of Lemma \ref{semidireto} and generalizes it.

\bigskip

\begin{proposicao} \label{corsemidireto}
	Let $\Gamma_+$ be a normal subgroup of index $2n$ of $\Gamma$ and $\gamma_i \in \Gamma \setminus \Gamma_+$ involutions, $i \in \{1,\cdots,n\}$. If $\Gamma$ can be decomposed as $$\Gamma = \Gamma_+ \rtimes (\mathbb{Z}_2(\gamma_1) \times \cdots \times \mathbb{Z}_2(\gamma_n)),$$ then $$\Gamma = (\Gamma_+ \rtimes \mathbb{Z}_2(\gamma_1) \rtimes \cdots \rtimes \mathbb{Z}_2(\gamma_{i})) \rtimes (\mathbb{Z}_2(\gamma_{i+1}) \times \cdots \times \mathbb{Z}_2(\gamma_n)),$$ 
	where $i \in \{1,\cdots,n\}$.
\end{proposicao}

\begin{proof}
	Let $\Gamma_+$ be a normal subgroup of $\Gamma$ and $\overline{\Gamma} =\Gamma_+ \rtimes \mathbb{Z}_2(\gamma_1) \rtimes \cdots \rtimes \mathbb{Z}_2(\gamma_{i})$. Note that $\overline{\Gamma}$ is also a normal subgroup of $\Gamma$. Let $\gamma \in \Gamma$ and $\bar{\delta} \in \overline{\Gamma}$ be any elements, with $\bar{\delta}=\delta \gamma^i$ and $\gamma^i=\gamma_1^{m_1}\gamma_2^{m_2}\cdots \gamma_i^{m_i} \in \mathbb{Z}_2(\gamma_1) \rtimes \cdots \rtimes \mathbb{Z}_2(\gamma_{i})$, where $m_j \in \{0,1\}$. Then
	\begin{eqnarray*}
		\gamma \bar{\delta} \gamma^{-1} = \gamma (\delta \gamma^i) \gamma^{-1} = (\gamma \delta \gamma^{-1})(\gamma \gamma^i \gamma^{-1})=\tilde{\delta}(\gamma \gamma^i \gamma^{-1}).
	\end{eqnarray*}
	Once $\tilde{\delta} \in \Gamma_+ \subset \overline{\Gamma}$, by hypothesis. We just show that $\gamma \gamma^i \gamma^{-1} \in \overline{\Gamma}$.
	Indeed, once that $\Gamma = \Gamma_+ \rtimes (\mathbb{Z}_2(\gamma_1) \times \cdots \times \mathbb{Z}_2(\gamma_n))$ so there is some $\delta_1 \in \Gamma_+$ such that $\gamma=\delta_1 \gamma^n$, with $\gamma^n \in \mathbb{Z}_2(\gamma_1) \times \cdots \times \mathbb{Z}_2(\gamma_n)$. Thus
	\begin{eqnarray*}
		\gamma \gamma^i \gamma^{-1}=\delta_1 \gamma^{n} \gamma^i (\delta_1 \gamma^{n})^{-1}=\delta_1 \gamma^{n} \gamma^i (\gamma^{n})^{-1} \delta_1^{-1}=\delta_1 \gamma^i \delta_1^{-1},
	\end{eqnarray*}
	because the group $\mathbb{Z}_2(\gamma_1) \times \cdots \times \mathbb{Z}_2(\gamma_n)$ is abelian, since each $\mathbb{Z}_2(\gamma_i)$ is also abelian (Lemma \ref{z2abeliano}). As $\delta_1$ and $\gamma^i$ belongs to $\overline{\Gamma}$, then $\gamma \gamma^i \gamma^{-1} \in \overline{\Gamma}$. Therefore, $\overline{\Gamma}$ is a normal subgroup of $\Gamma$. We claim that
	$$\Gamma=\overline{\Gamma} \rtimes (\mathbb{Z}_2(\gamma_{i+1}) \times \cdots \times \mathbb{Z}_2(\gamma_n)).$$
	From hypothesis, we have that each $\gamma \in \Gamma$ can be written in the form $\gamma=\delta \gamma^n$. Since $\bar{\delta}=\delta \gamma^i \in \overline{\Gamma}$ we have that $\gamma=\bar{\delta} \gamma_{i+1}^{m_{i+1}}\cdots \gamma_{n}^{m_n}$. Thus $\gamma \in \overline{\Gamma} \rtimes (\mathbb{Z}_2(\gamma_{i+1}) \times \cdots \times \mathbb{Z}_2(\gamma_n))$. The inclusion $\overline{\Gamma} \rtimes (\mathbb{Z}_2(\gamma_{i+1}) \times \cdots \times \mathbb{Z}_2(\gamma_n)) \subset \Gamma$ is clearly immediate. So it follows the equality of the sets. It is also clear that
	$$\overline{\Gamma} \cap (\mathbb{Z}_2(\gamma_{i+1}) \times \cdots \times \mathbb{Z}_2(\gamma_n))=\{e\},$$
	because each $\gamma_{j} \notin \overline{\Gamma}$, with $j=i+1,\cdots,n$.	
\end{proof}

The decomposition for $\Gamma$ seen above does not always hold. Firstly, the group $\Gamma$ must have at least one normal proper subgroup. Furthermore, an implicit condition is that the group $\Gamma$ has involutions. Borel's book deals with some conditions for the existence of this decomposition (\cite{borel1998}, Theorem 1.2). In that book, Borel defines $\Gamma$ as having many connected components and $\Gamma_+$ as a maximal compact subgroup.

However, there are other groups in which this decomposition occurs. In addition to the example presented at the end of this paper, consider the dihedral group $D_4$. This group has eight elements. The group $\langle R_{\pi} \rangle$ generated by rotating $\pi$ radians is a normal subgroup of $D_4$ with index four, and $D_4=R_{\pi} \rtimes \left(\mathbb{Z}_2(S_y) \times \mathbb{Z}_2(S_x)\right)$, where $S_x$ and $S_y$ are the reflections around the $x$ and $y$ axes, respectively. 
Note that group $\langle R_{\pi} \rangle$ is not maximal. In the case of the maximal compact group $H$ generated by all rotations, we have that $D_4=H \rtimes S_x$.

\bigskip

To simplify the notation, we take $i=(i_1,\cdots,i_n)$ an element of $N^n_2$, where $N_2$ is the set $\{0,1\}$. Let $\lambda=(\lambda_1,\cdots,\lambda_n)$, with $\lambda_i \in \Gamma$, we define the power $\lambda^i=\lambda_1^{i_1}\lambda_2^{i_2}\cdots \lambda_n^{i_n}$. In these terms, we enunciate the \textit{Fubini Theorem} for normal Lie subgroups of even index which is a particular case of the version presented in (\cite{brocker2013representations}, Proposition 5.16).

\bigskip

\begin{theorem}[Fubini Theorem for normal Lie subgroups of index $2n$] \label{genfubini}
	Let $\Gamma$ be a compact Lie group and $\Gamma_+ \subset \Gamma$ a normal subgroup of index $2n$ of $\Gamma$. For any continuous $f: \Gamma \to \mathbb{R}$ we have
	$$\int_{\Gamma}f(\gamma)=\frac{1}{2^n} \left(\sum_{i \in N^n_2} \int_{\Gamma_+}f(\lambda^i\gamma) \right),$$
	where $\lambda_i$'s are distinct involutions and $\lambda_i \in \Gamma_-$ are fixed.
\end{theorem}
\begin{proof}
We proceed by induction. For $n=1$ we have the classical Fubini theorem, Theorem \ref{fubini}. 
Assume that 
$$\int_{\bar{\Gamma}}f(\gamma)=\frac{1}{2^k} \left(\sum_{i \in N^k_2} \int_{\Gamma_+}f(\lambda^i\gamma) \right)$$
holds, where $\gamma_1,\cdots,\gamma_k \in \Gamma \setminus \Gamma_+$ are distinct involutions and $$\bar{\Gamma}=\Gamma_+ \rtimes \mathbb{Z}_2(\gamma_1) \rtimes \cdots \rtimes \mathbb{Z}_2(\gamma_{k}).$$
Suppose that
$$\Gamma=\Gamma_+ \rtimes (\mathbb{Z}_2(\gamma_1) \times \cdots \times \mathbb{Z}_2(\gamma_{k+1})),$$ where $\gamma_{k+1}$ is an involution and $\gamma_{k+1} \notin  \bar{\Gamma}$.
By Proposition \ref{corsemidireto} we have that $$\Gamma=(\Gamma_+ \rtimes \mathbb{Z}_2(\gamma_1) \rtimes \cdots \rtimes \mathbb{Z}_2(\gamma_{k})) \rtimes \mathbb{Z}_2(\gamma_{k+1})=\bar{\Gamma} \rtimes \mathbb{Z}_2(\gamma_{k+1}).$$
Note that $\bar{\Gamma}$ is a normal subgroup of index $2$ of $\Gamma$. Thus, we apply Theorem \ref{fubini} and we get
$$\int_{\Gamma}f(\gamma)=\frac{1}{2} \left(\int_{\bar{\Gamma}}f(\gamma)+\int_{\bar{\Gamma}}f(\lambda_{k+1} \gamma) \right).$$
Therefore, using the induction hypothesis, we have
\begin{eqnarray*}
\int_{\Gamma}f(\gamma)
&=&\frac{1}{2} \left(\frac{1}{2^k} \sum_{i \in N^k_2} \int_{\Gamma_+}f(\lambda^i\gamma) +\frac{1}{2^k} \sum_{i \in N^k_2} \int_{\Gamma_+} f(\lambda_{k+1} \lambda^i \gamma) \right) \\
&=& \frac{1}{2^{k+1}}\sum_{i \in N^k_2}  \left( \int_{\Gamma_+}f(\lambda_{k+1}^0\lambda^i\gamma)+ \int_{\Gamma_+} f(\lambda_{k+1} \lambda^i \gamma) \right) \\
&=& \frac{1}{2^{k+1}}\sum_{i \in N^k_2}  \left( \sum_{i_{k+1} \in N_2} \int_{\Gamma_+}f(\lambda_{k+1}^{i_{k+1}}\lambda^i\gamma) \right) \\
&=& \frac{1}{2^{k+1}} \sum_{j \in N^{k+1}_2} \int_{\Gamma_+}f(\lambda^j\gamma),
\end{eqnarray*}
once if $\lambda^j=\lambda^i\lambda_{k+1}^{i_{k+1}}$, with $i \in N^{k}_2$ and $i_{k+1} \in N_2$, then $j \in N^{k+1}_2$.
\end{proof}


The importance of this theorem must be now evident. To calculate a Haar integral over a Lie group $\Gamma$ is tantamount to calculate the integral over a normal subgroup of index $2n$ of $\Gamma$, which is naturally a simpler task.

\section{An application of Fubini theorem to invariant theory}

In this section we present an application of the Theorem \ref{genfubini} used to calculate the Molien series, a useful tool in invariant theory to estimate the number of invariant polynomials by the action of a Lie group on a vector space.

Consider the action of a Lie group $\Gamma$ on a vector space $V$. We say that a polynomial map $f:V \to \mathbb R$ is invariant by action of $\Gamma$ on $V$, or $\Gamma$-invariant, if
$$f(\gamma x)=f(x),$$
for all $x \in V$ and $\gamma \in \Gamma$. It is not difficult to show that the set of all the $\Gamma$-invariant polynomials has the structure of a ring, which we denote by $\mathcal{P}(\Gamma)$. It is known that if $\Gamma$ is a compact Lie group then the ring of invariant polynomials is finitely generated, this case is known as the Hilbert-Weyl Theorem (\cite{golubitsky2012singularities}). The result is also valid for the case where $\Gamma$ is a reductive group (\cite{luna} and \cite{derksen2015computational}). If the set of generators of the invariant polynomials is finite, it is called the \emph{Hilbert basis} of the ring $\mathcal{P}(\Gamma)$.


The ring $\mathcal{P}_V=\mathbb{R}[x_1,\cdots,x_n]$ is a graded algebra over $\mathbb{R}$, i.e.,
$$\mathcal{P}_V=\bigoplus_{d=0}^{\infty}\mathcal{P}_{V_d},$$ where $\mathcal{P}_{V_d}$ is the vector subspace of homogeneous polynomial maps of degree $d$ (\cite{sturmfels2008algorithms}). Let $\mathcal{P}_d(\Gamma)$ be the subspace of all homogeneous $\Gamma$-invariants polynomials of degree $d$. Since the action of $\Gamma$ in $V$ is linear then $\gamma f \in \mathcal{P}_d(\Gamma)$, for all $\gamma \in \Gamma$ and $f \in \mathcal{P}_d(\Gamma)$. Therefore, the ring of $\Gamma$-invariants polynomials  $\mathcal{P}(\Gamma)$ is a graded algebra over $\mathbb{R}$. This property is very useful because we can calculate the invariant polynomials degree to degree, since an invariant polynomial of degree $m$ can be written as a sum of invariant polynomials of degrees $1, 2,\cdots, m$.

	The Hilbert series of graded algebra $\mathcal{P}(\Gamma)$ is the function given by $$\Phi_{\Gamma}(t)=\sum_{d=0}^{\infty}\dim(\mathcal{P}_d(\Gamma))t^d.$$

Denoting by $c_d$ the number of invariant polynomials of degree $d$ on $V$, one defines the Hilbert series $$\Phi_\Gamma (t)=\sum \limits_{d=0}^{\infty} c_d t^d.$$

The following theorem is a classical result that gives us an explicit form of the Hilbert series in terms of the matrix representations of $\Gamma$. The Hilbert series defined in this way is called the \textit{Molien series}.

\bigskip

\begin{theorem} [Molien theorem]
	The Hilbert series of the invariant ring $\mathcal{P} (\Gamma)$ equals
	\begin{eqnarray*}
	\Phi_\Gamma (t)&=&
	\frac{1}{|\Gamma|} \sum_{\gamma \in \Gamma} \frac{1}{(\det(I-t\gamma))} ,\quad \mbox{if} \quad \Gamma \quad \mbox{is a finite discrete group and} \\
	\Phi_\Gamma (t)&=&
	\int_{\gamma \in \Gamma} \frac{1}{(\det(I-t\gamma))} ,\quad \mbox{if} \quad \Gamma \quad \mbox{is a compact group.} 
	\end{eqnarray*}
\end{theorem}
\begin{proof}
	See \cite{molien1897invarianten} for the original proof of the finite case, and \cite{sattinger2006group} for the extension to a compact group.
\end{proof}

Thus, the coefficients $c_d$ of the Molien series gives us an estimate of the amount of invariant polynomials of degree $d$.

In order to show an application of the Theorem \ref{genfubini}, let us use the invariant theory the action of the Lorentz group on Minkowski space.

By definition, the Lorentz group $O(n,1)$ is a subgroup of the Poincaré group that preserves the nondegenerate symmetric bilinear form 
$\left< , \right> :\mathbb{R}^{n+1}\times\mathbb{R}^{n+1}\longrightarrow \mathbb{R}$ defined by $$\left< x,y\right>=\sum_{i=1}^{n}x_{i}y_{i}-x_{n+1}y_{n+1},$$ called Lorentz inner pseudo-product and also preserve the Lorentz distance. However, the Lorentz group does not contain all the transformations that preserve the Lorentz distance, this is a property of the Poincaré group, the group of isometries in the Minkowski space. The space $\mathbb{R}^{n+1}$ endowed with the metric induced by Lorentz inner pseudo-product $\left< , \right>$ is called \textit{Minkowski space} and is denoted by $\mathbb{R}^{n+1}_1$.

\newpage

Consider the Lorentz subgroup 
$$\Gamma=\Gamma_+ \rtimes (\mathbb Z_2(\lambda_1) \times \mathbb Z_2(\lambda_2)) \subset O(3,1),$$
where
\begin{eqnarray*} 
\Gamma_+:=\left\{ \left( \begin{array}{cc} R_{\phi} & 0 \\ 0 & I_2 \end{array} \right): R_{\phi} \in SO(2) \right\} \label{grupocompacto}
\end{eqnarray*}
and $\mathbb Z_2(\lambda_1)$ and $\mathbb Z_2(\lambda_2)$ are generated by involutions of $O(3,1)$ given by 
$$\lambda_1:= \left( \begin{array}{cc} I_2 & 0 \\ 0 & S_h^-(\theta) \end{array} \right), \mbox{ with } S_h^-(\theta)= \left( \begin{array}{cc} \cosh(\theta) & \sinh(\theta) \\ -\sinh(\theta) & -\cosh(\theta) \end{array} \right)$$
and
$$\lambda_2:= \left( \begin{array}{cc} I_2 & 0 \\ 0 & S_h^+(\theta) \end{array} \right), \mbox{ with } S_h^+(\theta)= \left( \begin{array}{cc} -\cosh(\theta) & -\sinh(\theta) \\ \sinh(\theta) & \cosh(\theta) \end{array} \right),$$
where $\theta \in \mathbb{R}$ is fixed. Note that
\begin{enumerate}
	\item $\Gamma_+$ is a compact subgroup of $O(3,1)$ isomorphic to $SO(2) \times \{I_2\}$. 
	\item Once $\Gamma$ is abelian then $\Gamma_+$ is a normal subgroup of $\Gamma$.
	\item As $\Gamma=\Gamma_+ \cupdot \lambda_1 \Gamma_+ \cupdot \lambda_2 \Gamma_+ \cupdot \lambda_1\lambda_2 \Gamma_+$,
	where $$ \lambda_1\lambda_2=\left( \begin{array}{cc} I_2 & 0 \\ 0 & -I_2 \end{array} \right),$$ then $\Gamma/ \Gamma_+=\{I_4,-I_4,\lambda_1,\lambda_2\}$. So, $|\Gamma/ \Gamma_+|=4$.
\end{enumerate}
Therefore, $\Gamma_+$ is a normal compact subgroup of $\Gamma$ with index four.

Consider the standard action of $\Gamma$ in $\mathbb{R}^4_1$. We want to calculate the Molien series of $\Gamma$ given by 
$$\Phi_{\Gamma}(t)=\int_{\Gamma} \frac{1}{\det(I_4-t\gamma)}.$$
By Theorem \ref{genfubini}, the above integral can be written as
\begin{dmath*}
\Phi_{\Gamma}(t)=\frac{1}{4} \left( \int_{\Gamma_+} \frac{1}{\det(I_4-t\gamma)}+\int_{\Gamma_+} \frac{1}{\det(I_4-t\lambda_1 \gamma)}+\int_{\Gamma_+} \frac{1}{\det(I_4-t\lambda_2 \gamma)}+\int_{\Gamma_+} \frac{1}{\det(I_4-t\lambda_1 \lambda_2 \gamma)} \right).
\end{dmath*}
Note that
\begin{equation*}
	I_4-t\gamma=\left( 
	\begin{array}{cc} I_2 & 0 \\ 0 & I_2 \end{array} \right)-t\left( \begin{array}{cc} R_{\phi} & 0 \\ 0 & I_2 \end{array} \right)=\left( \begin{array}{cc} I_2-tR_{\phi} & 0 \\ 0 & I_2-tI_2 \end{array} \right),
\end{equation*}
therefore
$$\det(I_4-t\gamma)=\det(I_2-tR_{\phi})\det(I_2-tI_2)=(-2 t \cos \phi+t^{2}+1) (t-1)^{2}.$$
In the same way, we have
\begin{eqnarray*}
\det(I_4-t\lambda_1\gamma) &=& \det(I_2-tR_{\phi})\det(I_2-tS_h^-)=(2t \cos\phi-t^{2}+1)(t^2-1), \\
\det(I_4-t\lambda_2\gamma) &=& \det(I_2-tR_{\phi})\det(I_2-tS_h^+)=(2t \cos\phi-t^{2}+1)(t^2-1), \\
\det(I_4-t\lambda_1 \lambda_2 \gamma) &=& \det(I_2-tR_{\phi})\det(I_2+tI_2)= (-2 t \cos\phi+t^{2}+1)(t+1)^{2}.
\end{eqnarray*}
Thus, since the groups $\Gamma_+$ and $SO(2) \times \{I_2\}$ are isomorphic then we can use the Haar integral on $SO(2)$ presented in (\ref{haarcompacto}):
\begin{dmath*}
	\Phi_{\Gamma}(t)
	=\frac{1}{8\pi} \left(\int_{0}^{2 \pi} \frac{d\phi}{(-2 t \cos \phi+t^{2}+1) (t-1)^{2}}+
	\int_{0}^{2 \pi} \frac{d\phi}{(2t \cos\phi-t^{2}+1)(t^2-1)}+
	\int_{0}^{2 \pi} \frac{d\phi}{(2t \cos\phi-t^{2}+1)(t^2-1)}+
	\int_{0}^{2 \pi} \frac{d\phi}{(-2 t \cos\phi+t^{2}+1)(t+1)^{2}} \right)
	=\frac{1}{8\pi} \int_{0}^{2 \pi} \frac{4}{(t-1)^{2} (t+1)^{2} (-2 t \cos \phi+t^{2}+1)}d\phi .
\end{dmath*}
So, we use \textit{Maple} to calculate the Molien series
\begin{dmath*}
	\Phi_{\Gamma}(t)=1+3t^2+6t^4+10t^6+15t^8+21t^{10}+28t^{12}+36t^{14}+45t^{16}+O(t^{18}),
\end{dmath*}
which shows that there are three invariant polynomials of degree two, six  invariant polynomials of degree four, ten invariant polynomials of degree six and so on. Indeed, a Hilbert basis of the ring $\mathcal{P}(\Gamma)$ is given by $$\{x^2+y^2,z^2-t^2,((\cosh(\theta)-1) t+\sinh(\theta) z)^2\}.$$

{\bf Using Maple to calculate the Molien series:} We conclude this paper by presenting the steps used in Maple to calculate the Molien series
\begin{enumerate}
	\item Enable the package \textit{LinearAlgebra}: 
	
	$$with(LinearAlgebra);$$
	\item Define the matrices:
	
	\begin{eqnarray*}
		I_2= \left( \begin{array}{cc}
			1 & 0 	\\	0 & 1 
		\end{array}\right); \\
	R_{\phi}=\left( \begin{array}{cc}
		\cos \phi & -\sin \phi \\	\sin\phi & \cos \phi 
	\end{array} \right); \\
	S_h^-(\theta)= \left( \begin{array}{cc} \cosh(\theta) & \sinh(\theta) \\ -\sinh(\theta) & -\cosh(\theta) \end{array} \right); \\
	S_h^+(\theta)= \left( \begin{array}{cc} -\cosh(\theta) & -\sinh(\theta) \\ \sinh(\theta) & \cosh(\theta) \end{array} \right);
	\end{eqnarray*}
	\item Calculate the determinants of the matrices defined in the previous step
	\begin{eqnarray*}
		x &:=& simplify(Determinant(I_2-tR_{\phi})\cdot Determinant(I_2-tI_2)); \\
		y &:=& simplify(Determinant(I_2-tR_{\phi})\cdot Determinant(I_2-tS_h^-(\theta)));\\
		z &:=& simplify(Determinant(I_2-tR_{\phi})\cdot Determinant(I_2-tS_h^+(\theta)));\\
		w &:=& simplify(Determinant(I_2-tR_{\phi})\cdot Determinant(I_2-tS_h^-(\theta)S_h^+(\theta)));
	\end{eqnarray*}
	\item Simplify $\Phi_{\Gamma}(t)$:
	\begin{eqnarray*}
		simplify\left(\frac{1}{x} + \frac{1}{y} + \frac{1}{z} + \frac{1}{w}\right);
	\end{eqnarray*}
	After pressing "Enter", the following expression appears
	$$\frac{4}{(t-1)^{2} (t+1)^{2} (-2 t \cos(a)+t^{2}+1)}.$$
	\item Now, just calculate the integral:
	$$\mathit{simplify}\left(\frac{1}{8\pi}\int_{0}^{2\pi}\frac{4}{(t-1)^{2} (t+1)^{2} (-2 t \cos\phi+t^{2}+1)}d\phi\right);$$
	returning the following result
	$$\frac{\mathrm{csgn}(\frac{(t+1)^{2}}{\sqrt{(t-1)^{2} (t+1)^{2}}}) \mathrm{csgn}(t^{2}-1)}{(t-1)^{3} (t+1)^{3}},$$
	where the function $csgn$ is defined by
	\begin{eqnarray*}
		csgn(\psi(t))=\begin{cases}
			1, \mbox{ if } \Re(\psi(t))>0 \mbox{ or } \Re(\psi(t))=0 \mbox{ and } \Im(\psi(t))>0 \\
			-1, \mbox{ if } \Re(\psi(t))<0 \mbox{ or } \Re(\psi(t))=0 \mbox{ and } \Im(\psi(t))<0
		\end{cases},
	\end{eqnarray*}
	also $\Re(z)$ and $\Im(z)$ is a real part and imaginary part of a complex number $z$, respectively.
	\item Finally, we calculate the series:
	$$\mathit{series}(\frac{\mathrm{csgn}(\frac{(t+1)^{2}}{\sqrt{(t-1)^{2} (t+1)^{2}}}) \mathrm{csgn}(t^{2}-1)}{(t-1)^{3} (t+1)^{3}},t,18);$$
	and from there we get the following result:
	$$1+3 t^{2}+6 t^{4}+10 t^{6}+15 t^{8}+21 t^{10}+28 t^{12}+36 t^{14}+45 t^{16}+\mathrm{O}(t^{18}).$$
\end{enumerate}

\section*{Acknowledgements}
The authors are very grateful to the anonymous referee for valuable comments and suggestions.


\bibliographystyle{ufrr}
\bibliography{references}

\end{document}